\documentclass[10pt,draft]{article}
\usepackage{amsmath,amscd}
\usepackage{amssymb,latexsym,amsthm}
\usepackage{color}
\usepackage[spanish,english]{babel}
\usepackage{amssymb}
\usepackage{color}
\usepackage{amsmath,amsthm,amscd}
\usepackage[latin1]{inputenc}
\usepackage{lscape}
\usepackage{fancyhdr}
\usepackage{amsfonts}
\usepackage{pb-diagram}

\numberwithin{equation}{section}
\newtheorem{theorem}{Theorem}[section]

\newtheorem{definition}[theorem]{Definition}
\newtheorem{proposition}[theorem]{Proposition}
\newtheorem{lemma}[theorem]{Lemma}

\newtheorem{corollary}[theorem]{Corollary}

\theoremstyle{definition}

\title{\textbf{Universal property of skew $PBW$ extensions}}
\author{Juan Pablo Acosta
\\ Oswaldo Lezama\\
\texttt{jolezamas@unal.edu.co}
\\ Seminario de Álgebra Constructiva - SAC$^2$\\ Departamento de Matemáticas\\ Universidad Nacional de
Colombia, Sede Bogot\'a}
\date{}
\begin{document}
\maketitle
\begin{abstract}
\noindent In this paper we prove the universal property of skew $PBW$ extensions generalizing this
way the well known universal property of skew polynomial rings. For this, we will show first a
result about the existence of this class of non-commutative rings. Skew $PBW$ extensions include as
particular examples Weyl algebras, enveloping algebras of finite-dimensional Lie algebras (and its
quantization), Artamonov quantum polynomials, diffusion algebras, Manin algebra of quantum
matrices, among many others. As a corollary we will give a new short proof of the
Poincaré-Birkhoff-Witt theorem about the bases of enveloping algebras of finite-dimensional Lie
algebras.

\bigskip

\noindent \textit{Key words and phrases.} Skew polynomial rings, skew $PBW$ extensions, $PBW$
bases, quantum algebras.

\bigskip

\noindent 2010 \textit{Mathematics Subject Classification.} Primary: 16S10, 16S80. Secondary:
16S30, 16S36.
\end{abstract}

\section{Introduction}

Most of constructions in algebra are characterized by universal properties from which it is easy to
prove important results about the constructed object. This is the case of the universal property of
the tensor product; another well known example is the universal property for the localization of
rings and modules by multiplicative subsets. A key example in non-commutative algebra is the skew
polynomial ring $R[x;\sigma,\delta]$; the universal property in this case says that if $B$ is a
ring with a ring homomorphism $\varphi:R\to B$ and in $B$ there exists and element $y$ such that
$y\varphi(r)=\varphi(\sigma(r))y+\varphi(\delta(r))$ for every $r\in R$, then there exists an
unique ring homomorphism $\widetilde{\varphi}:R[x;\sigma,\delta]\to B$ such that
$\widetilde{\varphi}(x)=y$ and $\widetilde{\varphi}(r)=\varphi(r)$ (see \cite{McConnell}). In this
paper we prove the universal property of skew $PBW$ extensions generalizing the universal property
of skew polynomial rings. For this, we will prove first a theorem about the existence of skew $PBW$
extensions similar to the corresponding result on skew polynomial rings. As application we will get
the Poincaré-Birkhoff-Witt theorem about the bases of enveloping algebras of finite-dimensional Lie
algebras. This famous theorem says that if $K$ is a field and $\mathcal{G}$ is a finite-dimensional
Lie algebra with $K$-basis $\{y_1,\dots,y_n\}$, then a $K$-basis of the universal enveloping
algebra $\mathcal{U}(\mathcal{G})$ is the set of monomials $y_1^{\alpha_1}\cdots y_n^{\alpha_n}$,
$\alpha_i\geq 0$, $1\leq i\leq n$ (see \cite{Dixmier}, \cite{Humphreys}).

Skew $PBW$ extensions were defined firstly in \cite{LezamaGallego}, and their homological and
ring-theoretic properties have been studied in the last years (see \cite{Lezama-OreGoldie},
\cite{Chaparro}, \cite{lezamareyes1}, \cite{Venegas}). Skew polynomial rings of injective type,
Weyl algebras, enveloping algebras of finite-dimensional Lie algebras (and its quantization),
Artamonov quantum polynomials, diffusion algebras, Manin algebra of quantum matrices, are
particular examples of skew $PBW$ extensions (see \cite{lezamareyes1}). In this first section we
recall the definition of skew $PBW$ extensions and some very basic properties needed for the proof
of the main theorem.

\begin{definition}\label{gpbwextension}
Let $R$ and $A$ be rings. We say that $A$ is a \textit{skew $PBW$ extension of $R$} $($also called
a $\sigma-PBW$ extension of $R$$)$ if the following conditions hold:
\begin{enumerate}
\item[\rm (i)]$R\subseteq A$.
\item[\rm (ii)]There exist finite elements $x_1,\dots ,x_n\in A$ such $A$ is a left $R$-free module with basis
\begin{center}
${\rm Mon}(A):= \{x^{\alpha}=x_1^{\alpha_1}\cdots x_n^{\alpha_n}\mid \alpha=(\alpha_1,\dots
,\alpha_n)\in \mathbb{N}^n\}$.
\end{center}
In this case it says also that \textit{$A$ is a left polynomial ring over $R$} with respect to
$\{x_1,\dots,x_n\}$ and $Mon(A)$ is the set of standard monomials of $A$. Moreover, $x_1^0\cdots
x_n^0:=1\in Mon(A)$.
\item[\rm (iii)]For every $1\leq i\leq n$ and $r\in R-\{0\}$ there exists $c_{i,r}\in R-\{0\}$ such that
\begin{equation}\label{sigmadefinicion1}
x_ir-c_{i,r}x_i\in R.
\end{equation}
\item[\rm (iv)]For every $1\leq i,j\leq n$ there exists $c_{i,j}\in R-\{0\}$ such that
\begin{equation}\label{sigmadefinicion2}
x_jx_i-c_{i,j}x_ix_j\in R+Rx_1+\cdots +Rx_n.
\end{equation}
Under these conditions we will write $A:=\sigma(R)\langle x_1,\dots ,x_n\rangle$.
\end{enumerate}
\end{definition}
The following proposition justifies the notation and the alternative name given for the skew $PBW$
extensions.
\begin{proposition}\label{sigmadefinition}
Let $A$ be a skew $PBW$ extension of $R$. Then, for every $1\leq i\leq n$, there exists an
injective ring endomorphism $\sigma_i:R\rightarrow R$ and a $\sigma_i$-derivation
$\delta_i:R\rightarrow R$ such that
\begin{center}
$x_ir=\sigma_i(r)x_i+\delta_i(r)$,
\end{center}
for each $r\in R$.
\end{proposition}
\begin{proof}
See \cite{LezamaGallego}, Proposition 3.
\end{proof}

Observe that if $\sigma$ is an injective endomorphism of the ring $R$ and $\delta$ is a
$\sigma$-derivation, then the skew polynomial ring $R[x;\sigma,\delta]$ is a trivial skew $PBW$
extension in only one variable, $\sigma(R)\langle x\rangle$.

Some extra notation will be used in the rest of the paper.

\begin{definition}\label{1.1.6}
Let $A$ be a skew $PBW$ extension of $R$ with endomorphisms $\sigma_i$, $1\leq i\leq n$, as in
Proposition \ref{sigmadefinition}.
\begin{enumerate}
\item[\rm (i)]For $\alpha=(\alpha_1,\dots,\alpha_n)\in \mathbb{N}^n$,
$\sigma^{\alpha}:=\sigma_1^{\alpha_1}\cdots \sigma_n^{\alpha_n}$,
$|\alpha|:=\alpha_1+\cdots+\alpha_n$. If $\beta=(\beta_1,\dots,\beta_n)\in \mathbb{N}^n$, then
$\alpha+\beta:=(\alpha_1+\beta_1,\dots,\alpha_n+\beta_n)$.
\item[\rm (ii)]For $X=x^{\alpha}\in {\rm Mon}(A)$,
$\exp(X):=\alpha$ and $\deg(X):=|\alpha|$.
\item[\rm (iii)]If $f=c_1X_1+\cdots +c_tX_t$,
with $X_i\in Mon(A)$ and $c_i\in R-\{0\}$, then $\deg(f):=\max\{\deg(X_i)\}_{i=1}^t.$
\end{enumerate}
\end{definition}
The skew $PBW$ extensions can be characterized in a similar way as was done in
\cite{Gomez-Torrecillas2} for $PBW$ rings.
\begin{theorem}\label{coefficientes}
Let $A$ be a left polynomial ring over $R$ w.r.t. $\{x_1,\dots,x_n\}$. $A$ is a skew $PBW$
extension of $R$ if and only if the following conditions hold:
\begin{enumerate}
\item[\rm (a)]For every $x^{\alpha}\in {\rm Mon}(A)$ and every $0\neq
r\in R$ there exist unique elements $r_{\alpha}:=\sigma^{\alpha}(r)\in R-\{0\}$ and $p_{\alpha
,r}\in A$ such that
\begin{equation}\label{611}
x^{\alpha}r=r_{\alpha}x^{\alpha}+p_{\alpha , r},
\end{equation}
where $p_{\alpha ,r}=0$ or $\deg(p_{\alpha ,r})<|\alpha|$ if $p_{\alpha , r}\neq 0$. Moreover, if
$r$ is left invertible, then $r_\alpha$ is left invertible.

\item[\rm (b)]For every $x^{\alpha},x^{\beta}\in {\rm Mon}(A)$ there
exist unique elements $c_{\alpha,\beta}\in R$ and $p_{\alpha,\beta}\in A$ such that
\begin{equation}\label{612}
x^{\alpha}x^{\beta}=c_{\alpha,\beta}x^{\alpha+\beta}+p_{\alpha,\beta},
\end{equation}
where $c_{\alpha,\beta}$ is left invertible, $p_{\alpha,\beta}=0$ or
$\deg(p_{\alpha,\beta})<|\alpha+\beta|$ if $p_{\alpha,\beta}\neq 0$.
\end{enumerate}
\begin{proof}
See \cite{LezamaGallego}, Theorem 7.
\end{proof}
\end{theorem}

\section {Existence theorem for skew PBW extensions}

If $A=\sigma(R)\langle x_1,\dots,x_n\rangle$ is a skew $PBW$ extension of the ring $R$, then as was
observed in the previous section, $A$ induces unique endomorphisms $\sigma_i:R\to R$ and
$\sigma_i$-derivations $\delta_i:R\to R$, $1\leq i\leq n$. Moreover, by (\ref{sigmadefinicion2}),
there exist $c_{ij}, d_{ij}, a_{ij}^{(k)}\in R$ such that
$x_jx_i=c_{ij}x_ix_j+a_{ij}^{(1)}x_1+\cdots+a_{ij}^{(n)}x_n+d_{ij}$, with $1\leq i,j\leq n$.
However, note that if $i<j$, since $Mon(A)$ is a $R$-basis, then $1=c_{j,i}c_{i,j}$, i.e., for
every $1\leq i<j\leq n$, $c_{ji}$ is a right inverse of $c_{i,j}$ univocally determined. In a
similar way, we can check that $a_{ji}^{(k)}=-c_{ji}a_{ij}^{(k)}$, $d_{ji}=-c_{ji}d_{ij}$. Thus,
given $A$ there exist unique parameters $c_{ij}, d_{ij}, a_{ij}^{(k)}\in R$ such that
\begin{equation}
x_jx_i=c_{ij}x_ix_j+a_{ij}^{(1)}x_1+\cdots+a_{ij}^{(n)}x_n+d_{ij}, \ \text{for every}\  1\leq
i<j\leq n.
\end{equation}

\begin{definition}
Let $A=\sigma(R)\langle x_1,\dots,x_n\rangle$ be a skew $PBW$ extension. $\sigma_i,\delta_i,
c_{ij}$, $d_{ij}, a_{ij}^{(k)}$, $1\leq i<j\leq n$, defined as before, are called the parameters of
$A$.
\end{definition}

Conversely, given a ring $R$ and parameters $\sigma_i,\delta_i, c_{ij}$, $d_{ij}, a_{ij}^{(k)}$,
$1\leq i<j\leq n$, we will construct in this section a skew $PBW$ extension with coefficient ring
$R$ and satisfying the following equations
\begin{enumerate}
  \item For $i<j$ in $I$ and $k$ in $I$,
  $x_{j}x_{i}=c_{ij}x_{i}x_{j}+ \Sigma_{k}a_{ij}^{(k)} x_{k}+d_{ij}$,
  \item For $i\in I$ and $r\in R$, $x_{i}r=\sigma_{i}(r)x_{i}+\delta_i(r)$,
 \end{enumerate}
where $I:=\{1,\dots,n\}$.

\begin{definition}\label{complejidad}
 Let $R$ be a ring and $W$ be the free monoid in the alphabet $X\cup R$, with $X:=\{x_{i}:i\in I\}$.
 Let $w$ be a word of $W$, the complexity of $w$, denoted $c(w)$, is a triple of nonnegative integers $(a,b,c)$, where $a$ is
the number of $x$'s in $w$, $b$ is is the number of inversions involving only $x$'s, and $c$ is the
number of inversions of the type $(x_i,r)$.
\end{definition}
These triples are ordered with the lexicographic order, i.e., $(a,b,c)\leq(d,e,f)$ if and only if
$a<d$, or, $a=d$ and $b<e$, or, $a=d$, $b=e$ and $c\leq f$. This is a well order. Let $T$ be the
set of elements of $W$ such that $c(w)=(a,0,0)$ and $\mathbb{Z}T$ be the linear extension of $T$ in
$\mathbb{Z}\langle X\cup R\rangle$ (the $\mathbb{Z}$-free algebra in the alphabet $X\cup R$).

\begin{definition}\label{funcion-reduccion}
 Let $R$ be a ring, $\{c_{ij}\}_{i<j}$, $\{d_{ij}\}_{i<j}$ and
 $\{a_{ij}^{(k)}\}_{i<j,k}$ be elements of $R$  indexed by $i,j,k$ in $I$. Let $\sigma_{i},\delta_{i}:R
 \rightarrow R$ be two functions for each  $i\in I$. Suppose that $c_{ij}$ is left invertible and that
$\sigma_i(r)\neq 0$ for $r\neq 0$.  We define the function $p$
\begin{center}
 $p:W \rightarrow \mathbb{Z}\langle X\cup R\rangle$, with $X:=\{x_{i}:i\in I\}$,
\end{center}
 by induction in the complexity, as follows:
 \begin{enumerate}
 \item If $w\in T$ then $p(w)=w$.
 \item If $w=v_{1}x_{i}rv_{2}$, with $r\in R$, $v_1\in W$ and $rv_2\in T$ then
 \begin{center}
 $p(w)=p(v_{1}\sigma_{i}(r)x_{i}v_{2})+p(v_{1}\delta_{i}(r)v_{2})$.
 \end{center}
 \item If $w=v_{1}x_{j}x_{i}v_{2}$, where $v_1\in W$, $x_{i}v_{2}\in T$ with $i<j$, then
 \begin{center}
 $p(w)=p(v_{1}c_{ij}x_{i}x_{j}v_{2})+\Sigma_{k}p(v_{1}a_{ij}^{(k)}x_{k}v_{2})+p(v_{1}d_{ij}v_{2})$.
 \end{center}
\end{enumerate}
The linear extension of $p$ to $\mathbb{Z}\langle X\cup R\rangle\rightarrow \mathbb{Z}\langle X\cup
R\rangle$ is also denoted $p$. The image of $p$ is contained in $\mathbb{Z}T$. Let
$Mon:=\{\Pi_{k=1}^{n} x_{i_{k}}: i_{1}\leq \dots \leq i_{n}, n\geq0\}$, and $F_{R}(Mon)$ be the
left free $R-$module with basis $Mon$. We define $q:\mathbb{Z}T\rightarrow F_{R}(Mon)$ as the
bilinear extension of $q(r_{1}\dots r_{m}x_{i_{1}}\dots x_{i_{n}}):=
(\Pi_{k=1}^{m}r_{k})x_{i_{1}}\dots x_{i_{n}}$. Finally, we define $h:\mathbb{Z}\langle X\cup
R\rangle\rightarrow F_{R}(Mon)$ as $h:=qp$.
\end{definition}

\begin{theorem}[Existence]\label{pbw}
Let $R,I,X,a_{ij}^{k},c_{ij},\sigma_{i},\delta_{i},h,p,q$ be as in Definition
\ref{funcion-reduccion}. Then, there exists a skew $PBW$ extension $A$ of $R$ with variables
$X:=\{x_{i}:i\in I\}$ such that
 \begin{enumerate}
 \item[\rm (a)] $x_{i}r=\sigma_{i}(r)x_{i}+\delta_{i}(r)$.
  \item[\rm (b)] $x_{j}x_{i}=c_{ij}x_{i}x_{j}+\Sigma_{k}a_{ij}^{(k)}x_{k}+d_{ij}$, for $i<j$ in $I$.
 \end{enumerate}
if and only if
\begin{enumerate}
 \item[\rm (1)]For
 every $i$ in $I$, $\sigma_{i}$ is a ring endomorphism of $R$ and $\delta_{i}$ is $\sigma_{i}$-derivation.
\item[\rm (2)]$h(x_{j}x_{i}r)=h(p(x_{j}x_{i})r)$, for $i<j$ in $I$ and $r\in R$.
\item[\rm (3)]$h(x_{k}x_{j}x_{i})=h(p(x_{k}x_{j})x_{i})$, for $i<j<k$ in $I$.
\end{enumerate}
\end{theorem}
\begin{proof}
$\Rightarrow)$: Numeral (1) is the content of Proposition \ref{sigmadefinition}. Conditions (2) and
(3) follow from (a) and (b) and the associativity $x_j(x_ir)=(x_jx_i)r$ and
$x_k(x_jx_i)=(x_kx_j)x_i$.

$\Leftarrow)$: Define $t:F_R(Mon)\rightarrow \mathbb{Z}\langle X\cup R\rangle$ as $t(\Sigma
r_{\bar{x}}\bar{x}):=\Sigma r_{\bar{x}}\bar{x}\in \mathbb{Z}\langle X\cup R\rangle$, where $\Sigma
r_{\bar{x}}\bar{x}$ is the unique expression of an element in $F_R(Mon)$ as a sum over a finite
set, $\bar{x}\in Mon$ and $r_{\bar{x}}\neq 0$ is an element of $R$.

We define a product in $F_{R}(Mon)$ by
\begin{center}
$f\star g=h(t(f)t(g))$, $f,g\in F_{R}(Mon)$,
\end{center}
and we will prove in Lemma \ref{h-es-homo} below that $h(ab)=h(a)\star h(b)$, with
$a,b\in\mathbb{Z}\langle X\cup R\rangle$. From this we get that $h:\mathbb{Z}\langle X\cup
R\rangle\rightarrow F_R(Mon)$ is a surjection that preserves sums, products and $h(1)=1$. This
makes $F_R(Mon)$ a ring, which is a skew $PBW$ extension of $R$ by the definition of the product
$\star$.

To complete the proof we proceed to prove Lemma \ref{h-es-homo}, but for this, we have to show
first some preliminary propositions under the hypothesis (1)-(3).
\end{proof}

\begin{proposition}\label{suma-y-producto-de-escalares}
For $a,b\in W$ and $r,s\in R$ the following equalities hold:
\begin{enumerate}
\item[\rm (i)]$h(a0b)=0$.
\item[\rm (ii)]$h(a(-r)b)=-h(arb)$.
\item[\rm (iii)]$h(a(r+s)b)=h(arb+asb)$.
\item[\rm (iv)]$h(a1b)=h(ab)$.
\item[\rm (v)]$h(a(rs)b)=h(arsb)$.
\end{enumerate}
\end{proposition}
\begin{proof}
(i) and (ii) follow from (iii) since $r\mapsto h(arb)$ is a group homomorphism from the additive
group of $R$ into $F_{R}(Mon)$.

(iii) is proven by induction on $c(a(r+s)b)$ and applying the definition of $h$. Here the
conditions $\delta_i(a+b)=\delta_i(a)+\delta_i(b)$ and $\sigma_i(a+b)=\sigma_i(a)+\sigma_i(b)$ in
the hipothesis (1) of Theorem \ref{pbw} are used.

(iv) is proven by induction on $c(a1b)$ and making use of part (i). The relevant hypothesis are
$\sigma_i(1)=1$ and $\delta_i(1)=0$ which are part of the hypothesis (1) in Theorem \ref{pbw}.

(v) This part is proven by induction on $c(a(rs)b)$ and making use of (iii). The relevant
hypothesis are $\sigma_i(ab)=\sigma_i(a)\sigma_i(b)$ and
$\delta_i(ab)=\sigma_i(a)\delta_i(b)+\delta_i(a)b$.
\end{proof}

\begin{proposition}\label{q-modulo-h}
 Let $y,z\in\mathbb{Z}\langle X\cup R\rangle$ and $a\in\mathbb{Z}T$. Then $h(yaz)=h(ytq(a)z)$.
\end{proposition}
\begin{proof}
This is because we can obtain $tq(a)$ from $a$ with a finite number of operations described in
Proposition \ref{suma-y-producto-de-escalares}. Indeed if $a\in \mathbb{Z}T$ then by definition of
$T$ we heave $a=\Sigma n_{u}u$ where the sum is over $u\in T$, $n_u\in \mathbb{Z}$ and
$u=r_{1,u}\dots r_{m,u}x_{j_1}\dots x_{j_k}$ ($j_1,\dots j_k$ and $m,k$ depend on $u$) here
$r_{s}\in R$ and $1\leq j_1\leq\dots\leq j_k\leq n$. Then by definition of $t,q$ we have
$tq(a)=\Sigma_{x\in A} a(x)x$ where $A=\{x\in Mon(X):a(x)\neq 0\}$, and $a(x)=\Sigma_{u\in
B(x)}n_u\Pi_{s}r_{s,u}\in R$ where $B(x)=\{u\in T: x_{j_1}\dots x_{j_k}=x\}$. Using the Proposition
\ref{suma-y-producto-de-escalares} (i) we obtain that $h(ytq(a)z)=h(y\Sigma_{x\in Mon(X)} a(x)xz)$.
Using that $h$ is linear we get $h(y\Sigma_{x\in Mon(X)} a(x)xz)=\Sigma_{x\in Mon(x)}h(ya(x)xz)$.
Using Proposition \ref{suma-y-producto-de-escalares} (i),(ii),(iii) we get that
$h(ya(x)xz)=\Sigma_{u\in B(x)}n_uh(y(\Pi_{s}r_{s,u})xz)$. Using Proposition
\ref{suma-y-producto-de-escalares} (iv)(v) we get that $h(y(\Pi_{s}r_{s,u})xz)=h(yr_{1,u}\dots
r_{m,u}xz)=h(yuz)$.
\end{proof}
\begin{proposition}\label{independencia}
 If $x,y,z\in\mathbb{Z}\langle X\cup R\rangle$ then
 $h(xp(y)z)=h(xyz)$.
\end{proposition}
\begin{proof}
The identity is linear in $x,y,z$, so we may assume they are words. Next we proceed by induction on
$c(xyz)$.
First assume that the first inversion from right to left
in $xyz$ is in $y$, say $y=w_{1}x_{j}sw_{2}$ with $s=x_i$ with $i<j$ or $s\in R$, and $sw_2\in T$. Then
 $h(xyz)=h(xw_{1}p(x_{j}s)w_{2}z)=h(xp(w_{1}p(x_{j}s)w_{2})z)=h(xp(y)z)$ by the definition of $p$ and induction.\\
Now assume that the first inversion of $xyz$ is not contained in $yz$, or $xyz\in T$, in this case $y\in T$ and $p(y)=y$.\\
Now assume that the first inversion of $xyz$ is contained in $z$ say $z=w_1x_jsw_2$ with $sw_2\in T$ and $s=x_i$ with $i<j$ or $s\in R$.
Then $h(xyz)=h(xyw_1p(x_js)w_2)=h(xp(y)w_1p(x_js)w_2)=h(xp(y)z)$ by definition of $h$ and induction.\\
Now assume that the first inversion of $xyz$ has a part in $y$ and a part in $z$, say $y=y'x_j$ and $z=sz'$ with $z\in T$ and $s=x_i$ with $i<j$ or $s\in R$.
Assume further that the first inversion of $y$ exists and is contained in $y'$, say $y'=w_1x_ks'w_2$ with $s'w_2\in T$ an $s'=x_i$ with $i<k$ or $s'\in R$.
Then $h(xyz)=h(xy'p(x_js)z')=h(xp(y')p(x_js)z')=h(xp(w_1p(x_ks')w_2)p(x_js)z')=h(xw_1p(x_ks')w_2p(x_js)z')=h(xw_1p(x_ks')w_2x_jsz')=h(xp(w_1p(x_ks')w_2x_j)sz')=
h(xp(y)z)$ by definition of $h$ and induction applied alternatively.\\
So the last case is $y=y'x_kx_j$ with $k>j$ and $z=sz'$ with $s=x_i$ with $i<j$ or $s\in R$ and $z\in T$.
In this case
 $h(xyz)=h(xy'x_{k}p(x_{j}s)z')=h(xy'p(x_{k}p(x_{j}s))z')$ by definition of $h$ and induction, also observe
$h(xy'p(x_kp(x_js))z')=h(xy'p(p(x_{k}x_{j})s)z')$ because $qp(p(x_kx_j)s)=qp(x_kp(x_js)$ by hipothesis (2) and (3) in Theorem \ref{pbw}, and
also by Proposition \ref{q-modulo-h}. Also $h(xy'p(p(x_kx_j)s)z')=h(xy'p(x_kx_j)sz')=h(xp(y'p(x_kx_j))z)$ by induction applied twice,
and $h(xp(y'p(x_kx_j))z)=h(xp(y)z)$ by definition of $p$, as required.
\end{proof}
\begin{lemma}\label{h-es-homo}
 $h(ab)=h(a)\star h(b)$, for $a,b\in\mathbb{Z}\langle X\cup R\rangle$.
\end{lemma}
\begin{proof}
$h(a)\star h(b)=h(tqp(a)tqp(b))=h(p(a)p(b))=h(ab)$, the first equality is from the definition of
$\star$, the second equality is from Proposition \ref{q-modulo-h} twice
 and the third equality is Proposition
\ref{independencia} twice.
\end{proof}

\section{The universal property}

In this section we will prove the main theorem about the characterization of skew $PBW$ extensions
by a universal property in a similar way as this is done for skew polynomial rings. This problem
was studied in \cite{Acosta} where skew $PBW$ extensions were generalized to infinite sets of
generators.

\begin{theorem}[Main theorem: The universal property]\label{122}
Let $A=\sigma(R)\langle x_1,\dots,x_n\rangle$ be a skew $PBW$ extension with parameters
$\sigma_i,\delta_i, c_{ij}, d_{ij}, a_{ij}^{(k)}$, $1\leq i,j\leq n$. Let $B$ be a ring with
homomorphism $\varphi:R\to B$ and elements $y_1,\dots,y_n\in B$ such that
\begin{enumerate}
\item[\rm (i)]$y_i\varphi(r)=\varphi(\sigma_i(r))y_i+\varphi(\delta_i(r))$, for every $r\in R$.
\item[\rm (ii)]$y_jy_i=\varphi(c_{ij})y_iy_j+\varphi(a_{ij}^{(1)})y_1+\cdots
+\varphi(a_{ij}^{(n)})y_n+d_{ij}$.
\end{enumerate}
Then, there exists an unique ring homomorphism $\widetilde{\varphi}:A\to B$ such that
$\widetilde{\varphi}\iota=\varphi$ and $\widetilde{\varphi}(x_i)=y_i$, where $\iota$ is the
inclusion of $R$ in $A$.
\end{theorem}
\begin{proof}
Since $A$ is a free $R$-module with basis $Mon(A)$, we define the $R$-homomorphism
\begin{center}
$\widetilde{\varphi}:A\to B$, $r_1x^{\alpha_1}+\cdots+a_tx^{\alpha_t}\mapsto
\varphi(r_1)y^{\alpha_1}+\cdots+\varphi(a_t)y^{\alpha_t}$,
\end{center}
where $y^{\theta}:=y_1^{\theta_1}\cdots y_n^{\theta_n}$, with $\theta:=(\theta_1,\dots,\theta_n)\in
\mathbb{N}^n$. Note that $\widetilde{\varphi}(1)=1$.

$\widetilde{\varphi}$ is multiplicative: In fact, applying induction on the degree $|\alpha+\beta|$
we have
\begin{center}
$\widetilde{\varphi}(ax^{\alpha}bx^{\beta})=\widetilde{\varphi}(a[\sigma^{\alpha}(b)x^\alpha
x^\beta+p_{\alpha,b}x^\beta])=\widetilde{\varphi}[a\sigma^\alpha(b)[c_{\alpha,\beta}x^{\alpha+\beta}+p_{\alpha,\beta}]+ap_{\alpha,b}x^\beta]
=\varphi(a)\varphi(\sigma^{\alpha}(b))\varphi(c_{\alpha,\beta})y^{\alpha+\beta}+\varphi(a)\varphi(\sigma^{\alpha}(b))\varphi(p_{\alpha,\beta})(y)
+\varphi(a)\varphi(p_{\alpha,b})(y)y^\beta$,
\end{center}
where $\varphi(p_{\alpha,\beta})(y)$ is the element in $B$ obtained replacing each monomial
$x^\theta$ in $p_{\alpha,\beta}$ by $y^\theta$ and every coefficient $c$ by $\varphi(c)$. In a
similar way we have for $\varphi(p_{\alpha,b})(y)$ (observe that the degree of each monomial of
$p_{\alpha,b}x^\beta$ is $<|\alpha+\beta|$). On the other hand, applying (i) and (ii) we get
\begin{center}
$\widetilde{\varphi}(ax^{\alpha})\widetilde{\varphi}(bx^{\beta})=\varphi(a)y^\alpha\varphi(b)y^\beta=
\varphi(a)[\varphi(\sigma^{\alpha}(b))y^\alpha+\varphi(p_{\alpha,b})(y)]y^\beta=\varphi(a)\varphi(\sigma^\alpha(b))y^\alpha
y^\beta+\varphi(a)\varphi(p_{\alpha,b})(y)y^\beta=\varphi(a)\varphi(\sigma^\alpha(b))[\varphi(c_{\alpha,\beta})y^{\alpha+\beta}+\varphi(p_{\alpha,\beta})(y)]+
\varphi(a)\varphi(p_{\alpha,b})(y)y^\beta=\varphi(a)\varphi(\sigma^{\alpha}(b))\varphi(c_{\alpha,\beta})y^{\alpha+\beta}+\varphi(a)\varphi(\sigma^{\alpha}(b))\varphi(p_{\alpha,\beta})(y)
+\varphi(a)\varphi(p_{\alpha,b})(y)y^\beta$.
\end{center}

It is clear that $\widetilde{\varphi}\iota=\varphi$ and $\widetilde{\varphi}(x_i)=y_i$. Moreover,
note that $\widetilde{\varphi}$ is the only ring homomorphism that satisfy these two conditions.
\end{proof}

\begin{corollary}\label{2.3}
Let $R$ be a ring and $A=\sigma(R)\langle x_1,\dots,x_n\rangle$ be a skew $PBW$ extension of $R$
with parameters $\sigma_i,\delta_i, c_{ij}, d_{ij}, a_{ij}^{(k)}$, $1\leq i,j\leq n$. Let $B$ be a
ring with homomorphism $\varphi:R\to B$ and elements $y_1,\dots,y_n\in B$ such that the conditions
$($i$)$-$($ii$)$ in Theorem \ref{122} are satisfied with respect to the system of parameters
$\sigma_i,\delta_i, c_{ij}, d_{ij}, a_{ij}^{(k)}$, $1\leq i,j\leq n$, of the ring $R$. If $B$
satisfies the universal property, then $B\cong A=\sigma(R)\langle x_1,\dots,x_n\rangle$. Moreover,
the monomials $y_1^{\alpha_1}\cdots y_n^{\alpha_n}$, $\alpha_i\geq 0$, $1\leq i\leq n$ are a
$R$-basis of $B$.
\end{corollary}
\begin{proof}
By the universal property of $A$ there exists $\widetilde{\varphi}$ such that
$\widetilde{\varphi}\iota=\varphi$; by the universal property of $B$ there exists
$\widetilde{\iota}$ such that $\widetilde{\iota}\varphi=\iota$. Note that
$\widetilde{\iota}\widetilde{\varphi}\iota=\iota$ and
$\widetilde{\varphi}\widetilde{\iota}\varphi=\varphi$. The uniqueness gives that
$\widetilde{\iota}\widetilde{\varphi}=i_A$ and $\widetilde{\varphi}\widetilde{\iota}=i_B$.
Moreover, in the proof of Theorem \ref{122} we observed that $\widetilde{\varphi}$ is not only a
ring homomorphism but also a $R$-homomorphism, whence
\begin{center}
$\widetilde{\varphi}(Mon(A))=\{y_1^{\alpha_1}\cdots y_n^{\alpha_n}|\alpha_i\geq 0,1\leq i\leq n\}$
\end{center}
is a $R$-basis of $B$.
\end{proof}

\begin{corollary}
Let $R$ be a ring and $A=\sigma(R)\langle x_1,\dots,x_n\rangle$ be a skew PBW extension of $R$ with
parameters $\sigma_i,\delta_i, c_{ij}, d_{ij}, a_{ij}^{(k)}$, $1\leq i,j\leq n$. Let $B$ be a ring
that satisfies the following conditions with respect to the system of parameters
$\sigma_i,\delta_i, c_{ij}, d_{ij}, a_{ij}^{(k)}$, $1\leq i,j\leq n$, of the ring $R$.
\begin{enumerate}
\item[\rm (i)]There exists a ring homomorphism $\varphi:R\to B$.
\item[\rm (ii)]There exist elements $y_1,\dots, y_n\in B$ such that $B$ is a left free $B$-module with basis $Mon(y_1,\dots,y_n)$, and the product
is given by $r\cdot b:=\varphi(r)b$, $r\in R, b\in B$.
\item[\rm (iii)]The conditions $($i$)$ and $($ii$)$ in Theorem \ref{122} hold.
\end{enumerate}
Then $B\cong A=\sigma(R)\langle x_1,\dots,x_n\rangle$.
\end{corollary}
\begin{proof}
According to the universal property of $A$, there exists a ring homomorphism
$\widetilde{\varphi}:A\to B$ given by $r_1x^{\alpha_1}+\cdots+a_tx^{\alpha_t}\mapsto
\varphi(r_1)y^{\alpha_1}+\cdots+\varphi(a_t)y^{\alpha_t}$; from (ii) we get that
$\widetilde{\varphi}$ is bijective.
\end{proof}

\section{The Poincaré-Birkhoff-Witt theorem}

Using the results of the previous sections, we will give now a new short proof of the
Poincaré-Birkhoff-Witt theorem about the bases of enveloping algebras of finite-dimensional Lie
algebras. Recall that if $K$ is a field and $\mathcal{G}$ is a Lie algebra with $K$-basis
$Y:=\{y_1,\dots,y_n\}$, the enveloping algebra of $\mathcal{G}$ is the associative $K$-algebra
$\mathcal{U}(\mathcal{G})$ defined by $\mathcal{U}(\mathcal{G})=K\{y_1,\dots,y_n\}/I$, where
$K\{y_1,\dots,y_n\}$ is the free $K$-algebra in the alphabet $Y$ and $I$ the two-sided ideal
generated by all elements of the form $y_jy_i-y_iy_j-[y_j,y_i]$, $1\leq i,j\leq n$, where $[\,
,\,]$ is the Lie bracket of $\mathcal{G}$ (see \cite{McConnell}).

\begin{theorem}[Poincaré-Birkhoff-Witt theorem]
The standard monomials $y_1^{\alpha_1}\cdots y_n^{\alpha_n}$, $\alpha_i\geq 0$, $1\leq i\leq n$,
conform a $K$-basis of $\mathcal{U}(\mathcal{G})$.
\end{theorem}
\begin{proof}
For the ring $K$ we consider the following system of variables and parameters:
\begin{equation}\label{eq4.1}
X:=\{x_1,\dots,x_n\}, \sigma_i:=i_K, \delta_i:=0, c_{i,j}:=1, d_{ij}:=0,
[x_i,x_j]=a_{ij}^{(1)}x_1+\cdots+a_{ij}^{(n)}x_n, 1\leq i,j\leq n.
\end{equation}
We want to prove that conditions (1)-(3) in Theorem \ref{pbw} hold. Condition (1) trivially holds.
For (2) we have
\begin{center}
$h(x_jx_ir) = h(x_jrx_i) = h(rx_jx_i) = rx_ix_j + r[x_j,x_i]$;

$h(p(x_jx_i)r) = h(x_ix_jr)+h([x_j,x_i]r) = h(x_irx_j)+r[x_j,x_i] = rx_ix_j+r[x_j,x_i]$.
\end{center}

Condition (3) of Theorem \ref{pbw} also holds: In fact,

\begin{center}
$h(p(x_kx_j)x_i)=h(x_jx_kx_i)+h([x_k,x_j]x_i)=h(x_jx_ix_k)+h(x_j[x_k,x_i])+h([x_k,x_j]x_i)=x_ix_jx_k+h([x_j,x_i]x_k)+h(x_j[x_k,x_i])+h([x_k,x_j]x_i)=
x_ix_jx_k+(h(x_k[x_j,x_i])+h([[x_j,x_i],x_k]))+
(h([x_k,x_i]x_j)+h([x_j,[x_k,x_i]]))+(h(x_i[x_k,x_j])+h([[x_k,x_j],x_i]))=
h(x_kx_jx_i)+
h([[x_j,x_i],x_k]+[x_j,[x_k,x_i]]+[[x_k,x_j],x_i])=h(x_kx_jx_i)$.
\end{center}
The last equality holds by the Jacobi identity, the second to the last equality follows regrouping
the terms and applying the definition of $h$ to $h(x_kx_jx_i)$.

From Theorem \ref{pbw} we conclude that there exists a skew $PBW$ extension $A=\sigma(K)\langle
x_1,\dots,x_n\rangle$ that satisfies (\ref{eq4.1}), in particular, the monomials
$x_1^{\alpha_1}\cdots x_n^{\alpha_n}$, $\alpha_i\geq 0$, $1\leq i\leq n$, conform a $K$-basis of
$A$. But note that $\mathcal{U}(\mathcal{G})$ satisfies the hypothesis in Corollary \ref{2.3}, so
$\mathcal{U}(\mathcal{G})\cong A$ and $\mathcal{U}(\mathcal{G})$ has $K$-basis
$y_1^{\alpha_1}\cdots y_n^{\alpha_n}$, $\alpha_i\geq 0$, $1\leq i\leq n$.

\end{proof}



\end{document}